\documentclass{amsart}
\usepackage{graphicx}
\usepackage{amssymb,amscd,amsthm,amsxtra}
\usepackage{latexsym}
\usepackage{epsfig,esint}
\usepackage{verbatim}

\newtheorem{thm}{Theorem}[section]
\newtheorem{cor}[thm]{Corollary}
\newtheorem{lem}[thm]{Lemma}
\newtheorem{prop}[thm]{Proposition}
\theoremstyle{definition}
\newtheorem{defn}[thm]{Definition}
\theoremstyle{remark}
\newtheorem{rem}[thm]{Remark}
\numberwithin{equation}{section}

\newcommand{\R}{\mathbb R}

\newcommand{\be}{\begin{equation}}
\newcommand{\ee}{\end{equation}}

\newcommand{\eps}{\epsilon}

\newcommand{\p}{\partial}

\begin{document}

\title{Global solutions to nonlinear two-phase free boundary problems}

\author{D. De Silva}
\address{Department of Mathematics, Barnard College, Columbia University, New York, NY 10027, USA}
\email{\tt  desilva@math.columbia.edu}
\author{O. Savin}
\address{Department of Mathematics, Columbia University, New York, NY 10027, USA}\email{\tt  savin@math.columbia.edu}
\thanks{O.~S.~is supported by  NSF grant DMS-1200701.}

\begin{abstract}We classify global Lipschitz solutions to two-phase free boundary problems governed by concave fully nonlinear equations, as either two-plane solutions or solutions to  a one-phase problem. \end{abstract}

\maketitle
\section{Introduction}

The classical two-phase free boundary problem studies critical functions $u$ of the energy
$$J(u)=\int_\Omega |\nabla u|^2 + \mathcal H^n(\chi_{\{u>0\}}) \, dx,$$
which appears in various models of fluid mechanics or heat conduction (see for example \cite{ACF, LW}).

The critical functions are harmonic except on their $0$ level set where a jump condition on $\nabla u$ is imposed due to the presence of the discontinuous second term in the energy. Precisely, the Euler-Lagrange equation reads
\begin{equation}\label{2f0}
\left\{
\begin{array}{ll}
\triangle u=0, & \hbox{in $\{u \ne 0 \}$} \\
\  &  \\
(u_{\nu }^{+})^2- (u_\nu^-)^2=1, &
\hbox{on $\Gamma(u):= \{u=0\}.$} \\
\end{array}
\right. 
\end{equation}
Here $u_{\nu }^{+}$ and $u_{\nu }^{-}$ denote the normal derivatives in the
inward direction to the positivity sets of $u^+$ and $u^-$ respectively, and $\Gamma(u)$ is the \emph{free boundary} of $u$.

The Lipschitz continuity of solutions to the free boundary problem \eqref{2f0} was obtained by Alt, Caffarelli and Friedman in \cite{ACF}. They discovered a monotonicity formula relating two nonnegative harmonic functions $u^\pm$ defined in disjoint domains $\Omega^\pm$ in say $B_2$, that vanish on the boundary $\p \Omega^\pm \cap B_1$. Precisely, the formula states that
$$\Phi(r):= \left( \frac{1}{r^2} \int_{B_r} \frac{|\nabla u^+|^2}{|x|^{n-2}} dx\right) ^\frac 12   \left( \frac{1}{r^2} \int_{B_r} \frac{|\nabla u^-|^2}{|x|^{n-2}} dx\right) ^ \frac 12$$
is monotone increasing in $r$, and $\Phi$ is constant if and only if $u^+$ and $u^-$ are linear functions. In the context of \eqref{2f0},  this formula implies that the energy at scale 1 bounds the product of the slopes at a free boundary point $0 \in \Gamma(u),$
$$ \Phi(1) \ge \Phi(0+) \, \simeq \, u_\nu^+(0) \cdot u_\nu^-(0),$$ 
and then it follows from the free boundary condition  that $u^+_\nu(0)$ is bounded above. 

The Lipschitz continuity of solutions to \eqref{2f0} is a crucial ingredient also in the study of the regularity of $\Gamma(u)$. Indeed,  after a blow-up analysis, the question of the regularity of $\Gamma(u)$ can be reduced to the classification of global Lipschitz solutions (see Caffarelli \cite{C2}). Another consequence of the ACF monotonicity formula (by letting $ r \to \infty$) is the classification of global ``purely two-phase" solutions (i.e. with $u^- \not \equiv 0$) as the trivial two-plane solutions $$p_{a,b}:=a (x\cdot \nu)^+ - b (x\cdot \nu)^-, \quad \quad a^2-b^2=1, \quad \quad a, b >0,$$
for some unit direction $\nu$. This implies that the only types of singularities for the free boundary $\Gamma(u)$ are the ones that occur in the one-phase setting when $u^- \equiv 0$.

The ACF monotonicity formula has been extensively used in various other contexts, however it is specific to the Laplace operator (see also \cite{CJK, MP}). In \cite{DS} and \cite{CDS} we investigated the questions of Lipschitz regularity of solutions for two-phase free boundary problems governed by fully nonlinear operators $\mathcal F(D^2u)$ and with a general isotropic free boundary condition $$u_\nu^+=G(u_\nu^-).$$

In \cite{DS} we obtained the Lipschitz continuity of solutions under the assumption that $G(t)$ behaves like $t$ for all $t$ large, and this condition is satisfied for example in the two-phase problem \eqref{2f0}. If in addition $\mathcal F$ is homogenous, it suffices to require $G(t)/t \to c_0$ as $t \to \infty$ for some constant $c_0$.

It turns out that in dimension $n=2$ the results can be improved significantly. This was shown by the authors in collaboration with Caffarelli in \cite{CDS} where
the Lipschitz continuity and the classification of global ``purely two-phase" solutions were obtained for linear equations with measurable coefficients under very general free boundary conditions $u_\nu^+=G(u_\nu^-,\nu,x)$.

In this paper we continue the study of the classification of global Lipschitz two-phase solutions in dimension $n \ge 3$. As mentioned above this problem is intimately connected to the regularity of the free boundary $\Gamma(u)$. In the nonlinear case the $C^{1,\alpha}$ regularity of $\Gamma(u)$ under perturbative assumptions was obtained by  several authors, in slightly different settings (see for example \cite{DFS, F1, Fe1,W1, W2}.)

We consider the two-phase free boundary problem governed by a fully nonlinear uniformly elliptic operator  $\mathcal{F}$
\begin{equation}
\left\{
\begin{array}{ll}
\mathcal{F}(D^{2}u)=0, & \hbox{in $B_1^+(u) \cup B_1^-(u),$} \\
\  &  \\
u_{\nu }^{+}=G(u_{\nu }^{-}), &
\hbox{on $\Gamma(u):= \partial
B^+(u) \cap B_1.$} \\
\end{array}
\right.  \label{fb}
\end{equation}%
Here $B_r \subset \mathbb{R}^{n}$ denotes the ball of radius $r$ centered at $0$ and 
\begin{equation*}
B_1^{+}(u):=\{x\in B_1 :u(x)>0\},\quad B_1^{-}(u):=\{x\in B_1
:u(x)\leq 0\}^{\circ },
\end{equation*}%
while as noted above, $u_{\nu }^{+}$ and $u_{\nu }^{-}$ denote the normal derivatives in the
inward direction to $B_1^{+}(u)$ and $B_1^{-}(u)$ respectively. 
The function $G: \R^+  \to \R^+$ is of class $C^1$ and it satisfies the usual ellipticity assumption, that is $G(t)$ is strictly increasing.

Our main result is a Liouville theorem for global ``two-phase" Lipschitz solutions. 

\begin{thm}\label {T2}
Let $u$ be a Lipschitz continuous viscosity solution to \eqref{fb} in $\R^n$.
Assume that \be\label{F} \text{$\mathcal F$ is concave (or convex) and $\mathcal F$ is homogeneous of degree 1.}
\ee

Then either $u$ is a two plane-solution 
\begin{equation}\label{2p}
u=a ((x-x_0)\cdot \nu)^+ - b ((x-x_0) \cdot \nu)^- \quad \quad \mbox{with} \quad a, b>0,  a=G(b), 
\end{equation}
or 
\begin{equation}\label{1p}
u^- \equiv 0,
\end{equation} 
which means that $u$ solves the one-phase problem for $\mathcal F$.
\end{thm}

Theorem \ref{T2} can be extended to more general operators $\mathcal F(D^2 u,\nabla u , u)$ that depend also on $\nabla u$ and $u$ if appropriate assumptions are imposed on $\mathcal F$. Here we only state a version that applies to the $p$-Laplace equation. We consider quasilinear equations of the type
\be\label{pL}
a_{ij}\left(\frac{\nabla u}{|\nabla u|} \right) \, \, u_{ij}=0,\ee
with uniformly elliptic coefficients $a_{ij} \in C^1(S^{n-1})$. We remark that in \cite{DS} we established also the Lipschitz continuity of solutions to the two-phase free boundary problem governed by \eqref{pL} (see also the paper of Dipierro and Kharakhian \cite{DK} for the case of minimizers in the $p$-Laplace equation). 

\begin{thm}\label {T2.1}
The conclusion of Theorem $\ref{T2}$ holds if $u$ is a global Lipschitz continuous viscosity solution to \eqref{fb} governed by equation \eqref{pL} instead of $\mathcal F$.
\end{thm}

We roughly outline a formal idea of the proof of Theorem \ref{T2}. Notice that $|\nabla u^{+}|$ is a subsolution of the linearized equation in $\{u>0\}$ and then its supremum occurs on the boundary $\Gamma(u)$. Assume for the moment that $\Gamma(u)$ is of class $C^2$ and that  $|\nabla u^{+}|$ achieves its maximum at a point, say at $0 \in \Gamma(u)$. Suppose the normal to $\Gamma(u)$ at 0
(pointing in the positive phase) is $e_n$ and let $\kappa_1, \ldots, \kappa_{n-1}$ be the principal curvatures of $\Gamma(u)$ at $0$.
 Then the mixed derivatives vanish, i.e. $u^+_{in}(0)=0$, $i<n$,
 $$D^2u^{+}(0)=diag(\kappa_1u_n^{+}(0), \ldots, \kappa_{n-1}u^{+}_n(0), u_{nn}^{+}(0))$$
and similarly (in view of the free boundary condition)
 $$D^2u^{-}(0)=diag(-\kappa_1u_n^{-}(0), \ldots, -\kappa_{n-1}u^{-}_n(0), u_{nn}^{-}(0)).$$
Using that $\mathcal F(D^2u^+(0))= \mathcal F(-D^2u^-(0))=0$ and that $\mathcal F$ is homogenous of degree one we have
$$\mathcal F \left(diag \left(\kappa_1, \ldots, \kappa_{n-1}, \frac{u^+_{nn}(0)}{u_n^+(0)}\right) \right)= \mathcal F\left(diag \left(\kappa_1, \ldots, \kappa_{n-1}, -\frac{u^-_{nn}(0)}{u_n^-(0)} \right) \right)=0.$$
which by ellipticity of $\mathcal F$ gives $$\frac{u^+_{nn}(0)}{u_n^+(0)}= -\frac{u^-_{nn}(0)}{u_n^-(0)}.$$
On the other hand, by Hopf lemma applied to $u_n^+$, $u_n^-$ we have $$u_{nn}^{+}(0), u_{nn}^{-}(0) <0,$$ (unless $u_n^+$ and $u_n^-$ are constant) 
and we reach a contradiction.

The rigorous proof requires somewhat involved and technical arguments. One of the main steps consists in obtaining a weak Evans-Krylov type estimate for a nonlinear transmission problem. The assumptions that $\Gamma(u)$ is of class $C^2$ (even if $\mathcal F$ is concave) and that  $|\nabla u^{+}|$ achieves its maximum at a point (rather than at infinity) cannot be justified. One major difficulty is that $\Gamma(u)$ is not known to be better than $C^{1,\alpha}$ even in the perturbative setting.

The idea of proof of Theorem \ref{T2} is to show a ``reversed" improvement of flatness for the solution $u$, which means that if $u$ is sufficiently close to a two plane solution at a small scale then it remains close to the same two-plane solution at all larger scales.
The key ingredient in the proof of our main Theorem \ref{T2} is the Proposition \ref{Nonlinear_intro} below. 

Since $\mathcal F$ is homogeneous of degree one, we can multiply $u^+$ and $u^-$ by suitable constants and assume that $G(1)=1$ and $\|\nabla u\|_{L^\infty}=1$. Denote by $P_{M,\nu}$ quadratic approximations of slope 1 to our free boundary problem \eqref{fb},
\be\label{W_intro}  P_{M,\nu}(x):=x \cdot \nu + \frac 1 2 x^T M x,\ee
with $\nu$ a unit direction and 
$$\text{$M \in S^{n\times n}$ such that $M\nu=0$ and $\mathcal F(M)=0$}.$$
Throughout the paper constants depending on $n,$ the ellipticity constants $\lambda, \Lambda$ of $\mathcal F$, and the modulus of continuity $\omega$ of $G'$ on $[0,2]$, are called universal constants.

Proposition \ref{Nonlinear_intro} is a dichotomy result for solutions $u$ which are $\eps$-perturbations of polynomials $P_{M,\nu}$ at scale 1. It says that either $u$ can be approximated by another polynomial $P_{\bar M, \bar \nu}$ in a $C^{2,\alpha}$ fashion at a smaller scale, or that $|\nabla u(0)|$ has to be strictly below 1 an amount of order $\eps$. 

\begin{prop}[Nonlinear Dichotomy]\label{Nonlinear_intro} Assume that $0 \in F(u)$, $G(1)=1$ and $|\nabla u| \leq 1$.  There exist small universal constants, $\eps_0, \delta_0, r_0, c_0, \alpha_0 >0$ such that if 
\be \label{flat1}
|u(x) - P_{M, e_n}(x)|  \leq \eps \quad \text{in $B_1$}, \quad \eps \leq \eps_0
\ee
with
\be\label{hyp}
\|M\| \leq \delta_0 \eps^{1/2}, 
\ee
then one of the following alternatives holds:

(i) \be\label{flat2} |u - P_{\bar M, \nu}| \leq \eps r_0^{2+\alpha_0}, \quad \text{in $B_{r_0}$}\ee
for some $\bar M, \nu$ with $$\|M-\bar M\| \leq C \eps$$ and $C$ universal;

(ii)\be\label{flat3} |\nabla u^+(0)| \leq 1-c_0 \eps.\ee
\end{prop}

We remark that assumption \eqref{flat1} implies that $\Gamma(u) \in C^{1,\alpha}$ and $u$ is a classical solution, hence $\nabla u^+ (0)$ is well defined (see Section 2).

The paper is organized as follows. In Section 2 we provide notation and definitions, and we recall the flatness result established in \cite{DFS}. The proof of the main Theorem \ref{T2} is presented in Section 3, assuming that Proposition 
\ref{Nonlinear_intro} holds. In Section 4 we study the transmission problem which appears as the linearization to our free boundary problem, and we prove a linear version of Proposition \ref{Nonlinear_intro}. In the last section we provide the proof of Proposition \ref{Nonlinear_intro} and we also sketch a proof of Theorem \ref{T2.1}.


\section{Preliminaries}


In this section we present some preliminary definitions and known results. 

First we define viscosity solutions to \eqref{fb}.
Recall that $\mathcal F: \mathcal{S}^{n\times n} \to \R$ is uniformly elliptic if 
there exist $0<\lambda\leq \Lambda$ such that for every  $M, N\in \mathcal{S}^{n\times n},$ 
$$
\mathcal M_{\lambda,\Lambda}^-(N) \leq \mathcal{F}(M+N)-\mathcal{F}(M)\leq \mathcal M_{\lambda,\Lambda}^+(N).
$$
Here $\mathcal{S}^{n\times n}$ denotes the set of real  $n\times n$ symmetric matrices and $\mathcal M^\pm_{\lambda,\Lambda}$ the extremal Pucci operators
$$\mathcal M^-_{\lambda,\Lambda}(N)= \Lambda \sum_{\mu_i<0} \mu_i +  \lambda \sum_{\mu_i>0} \mu_i , \quad \quad \quad
\mathcal M^+_{\lambda,\Lambda}(N)= \lambda \sum_{\mu_i<0} \mu_i +  \Lambda \sum_{\mu_i>0} \mu_i,$$
with $\mu_i$ denoting the eigenvalues of $N$.

The class of all uniformly elliptic operators with ellipticity constants $\lambda,\Lambda$ and such that $\mathcal F(0)=0$ will be denoted by $\mathcal{E}(\lambda, \Lambda).$

Given $u, \varphi \in C(B_1)$, we say that $\varphi$
touches $u$ by below (resp. above) at $x_0 \in B_1$ if $u(x_0)=
\varphi(x_0),$ and
$$u(x) \geq \varphi(x) \quad (\text{resp. $u(x) \leq
\varphi(x)$}) \quad \text{in a neighborhood $O$ of $x_0$.}$$ If
this inequality is strict in $O \setminus \{x_0\}$, we say that
$\varphi$ touches $u$ strictly by below (resp. above).

Let $\mathcal F \in \mathcal E(\lambda, \Lambda)$. If $v \in C^2(O)$, $O$ open subset in $\R^n,$ satisfies  $$\mathcal F(D^2 v) > 0  \  \ \ (\text{resp}. <0)\quad \text{in $O$,}$$ we call $v$ a (strict) classical subsolution (resp. supersolution) to the equation $\mathcal F(D^2 v) = 0 $ in $O$.

We recall that $u \in C(O)$ is a viscosity solution to 
\be\label{Fv}
\mathcal F(D^2 u) = 0  \quad \text{in $O$,}
\ee 
if $u$ cannot be touched by  below (resp. above) by a strict classical subsolution (resp. supersolution) at a point $x_0 \in O.$ Similarly we can define viscosity subsolutions/supersolutions.

We refer the reader to \cite{CC} for  a comprehensive treatment of the theory of fully nonlinear elliptic equations. In particular, if $u$ is a viscosity solution to \eqref{Fv} then $u \in C^{1,\alpha}$ and any directional derivative $u_e$ belongs to the class ${\mathcal{S}}(\lambda,\Lambda)$ of ``solutions" to linear equation with measurable coefficients, and therefore satisfies the Harnack inequality. 

If in addition $\mathcal F$ is concave then $u \in C^{2,\alpha}$ by the Evans-Krylov theorem, and any second directional derivative $u_{ee}$ belongs to the class of subsolutions  $\underline{\mathcal{S}}(\lambda,\Lambda)$  to linear equation with measurable coefficients, and therefore it satisfies the Weak Harnack inequality.

We now turn to the free boundary condition. 

\begin{defn}\label{def} We say that  $u$ satisfies the free boundary condition 
$$ u_\nu^+=G(u_\nu^-), $$
at a point $y_0 \in \Gamma(u)$ if for any unit vector $\nu$, there exists no function $\psi \in C^2 $ defined in a neighborhood of $y_0$ with $\psi(y_0)=0$, $\nabla \psi(y_0)=\nu$ such that either of the following holds:

(1) $a \psi^+ - b \psi^- \le u$ with  $a>0$, $b > 0$  and $a> G(b)$ (i.e. $u$ is a supersolution);

(2) $a \psi^+ - b \psi^- \ge u$ with $a>0$, $b > 0$ and $a < G(b)$ (i.e. $u$ is a subsolution).
\end{defn}

We only use comparison functions which cross the $0$ level set transversally and therefore have a nontrivial negative part. For this reason the free boundary condition is preserved when taking uniform limits. It is straightforward to check that a uniform limit of solutions of \eqref{fb} satisfies \eqref{fb} as well (see Lemma 2.3 in \cite{DS}).

Next we state the flatness theorem obtained by De Silva, Ferrari and Salsa in \cite{DFS}.

\begin{defn}A {\it two-plane solution} $U$ to \eqref{fb} is given by
$$U(x)=U_{a^\pm, \nu}(x):=a^+(x\cdot \nu)^+- a^-(x \cdot \nu)^-,$$
for some $\nu \in S^1$ and with $$a^+=G(a^-), \quad \quad \mbox{and} \quad a^+>0, a^->0.$$\end{defn}


The following result was established in \cite{DFS}. 

\begin{thm}[DFS]\label{flatness} Let $u$ be a viscosity solution to \eqref{fb} satisfying \be\label{flat} |u(x)-U(x)| \leq \eps \quad  \text{in $B_1$}, \quad  0 < a^-_0 \leq a^- \leq a^-_1.\ee There exists a constant $\bar \eps(a_0^-,a_1^-)$ such that if $\eps \leq \bar \eps$
then $\Gamma(u)$ is $C^{1, \alpha}$ in $B_{1/2}$, and the $C^{1,\alpha}$ norm of $\Gamma(u)$ is bounded by a universal constant.
\end{thm}
We need a refinement of Theorem \ref{flatness} in which the approximation of $u$ in a $C^{1,\alpha}$ fashion is done by using two-phase quadratic polynomials rather than two-plane solutions. 
For this we first introduce the family $V_{a^\pm, M, \nu}$ of quadratic approximations with general slopes $a^+,a^- \in [1/2, 2] $ as
\be\label{V} V_{a^\pm, M, \nu}(x):= a^+P_{M,\nu}^+(x) - a^-P_{M,\nu}^-(x), \quad P_{M,\nu}(x):=x \cdot \nu + \frac 1 2 x^T M x,\ee
with $\nu$ a unit direction and 
$$a^+ = G(a^-), \quad \text{$M \in S^{n\times n}$ such that $M \nu=0$ and $\mathcal F(M)=0$}.$$
Notice that 
\be\label{W}V_{1,1, M\nu}=P_{M, \nu},\ee
while
$$V_{a^\pm, 0, \nu}= U_{a^\pm, \nu}.$$
Next we approximate $u$ by elements of the family $V$ with quadratic part $M$ of order $\eps^{1/2} \gg \eps$.


 \begin{prop}\label{c1} Assume that $0 \in \Gamma(u),$ and $\mathcal F$ satisfies \eqref{F}. There exists a universal constant $r_1>0$, such that if $u$ satisfies 
\be \label{flatflat}
|u - V_{a^\pm, M, e_n}|  \leq \eps \quad \text{in $B_1$}, 
\ee
with
\be\label{hyp2}
\|M\| \leq \delta \eps^{1/2},
\ee
for some $0< \eps \leq \eps_1$, then 
 \be\label{flat2flat2} |u - V_{\bar a^\pm, \bar M, \bar\nu}| \leq \eps r_1^{1+\alpha_1} \quad \text{in $B_{r_1}$}\ee
for some $\alpha_1$ universal, with $$|\bar a^+-a^-|, |\bar \nu-e_n|, \|M-\bar M\| \leq C\eps,$$ and $C$ universal.
\end{prop}

The proof of Proposition \ref{c1} is postponed till Section 5. We will follow the arguments in \cite{DFS} where the same result was obtained without the presence of the quadratic part, i.e. in the case $M=0$, $\bar M=0$.

Notice that the conclusion of Proposition \ref{c1} can be iterated indefinitely. Indeed, if $u$ satisfies \eqref{flat2flat2}, then the rescaling 
$$u_r(x):= \frac{u(rx)}{r}, \quad \quad r=r_1,$$
will satisfy
$$|u - V_{\bar a^\pm, r\bar M, \bar \nu}|\leq \eps_r:= \eps r^{\alpha_1}, \quad \text{in $B_1$}.$$
In order to apply Proposition \ref{c1} once again, we need to check that
$$\|r\bar M\|\leq \delta \eps_r^{1/2}, \quad \eps_r \leq \eps \leq \eps_1.$$
This is clearly satisfied since (by choosing $\eps_1$ possibly smaller)
$$\|r \bar M\| \leq \|rM\| + C \eps r \le r \delta \eps^{1/2} + C \eps r \leq \delta \eps_r^{1/2}= \delta (\eps r^{\alpha_1})^{1/2}.$$
A consequence of Proposition \ref{c1} is the corollary below which is a refined version of Theorem \ref{flatness}.

\begin{cor}\label{maincor}There exist universal constants, $\eps_1, \delta>0$ such that if \eqref{flatflat}-\eqref{hyp2} hold with $0<\eps \leq \eps_1, $ then $\Gamma(u)$ has small $C^{1,\alpha}$ norm in $B_{1/2}$ and 
\be\label{cor1}
|\nabla u^+(0)| \leq a^+ + C\eps
\ee
\be\label{cor2}\left |\frac{\nabla u(0)}{|\nabla u(0)|} - e_n\right | \leq C \eps 
\ee
with $C>0$ universal.
\end{cor}

\section{The proof of Theorem \ref{T2}}

In this section we will provide the proof of our main Theorem \ref{T2}. In Lemma \ref{uxn} we show that $u$ can be well-approximated by a two-plane solution with maximal slopes at some scale, possibly small. Then we use Proposition \ref{Nonlinear_intro} to obtain a ``reversed" improvement of flatness property and conclude that the solution $u$ is well approximated by the same two-plane solution at all large scales.

Let $u$ be a Lipschitz viscosity solution to \eqref{fb} in $\R^n$.
Call,
\be\label{ab}a:= \sup_{\{u>0\}}|\nabla u|, \quad b:=\sup_{\{u < 0\}}|\nabla u|.\ee

The next lemma provides an initial flatness condition for $u$ at some arbitrary scale.

\begin{lem}\label{uxn} Assume that $a, b>0.$Then, $$a=G(b).$$ Moreover, there exists a sequence $u_k(x):=\frac{u(x_k+d_k x)}{d_k}$ of rescalings of $u$ in $B_2,$ with 
$$u_k \to ax_n^+-bx_n^-$$ uniformly on compacts of $B_2.$
\end{lem}
\begin{proof} Assume that
\be \label{ab2} a>G(b).\ee
Given $\eps>0$, let $x_\eps \in \{u>0\}$ be such that
$$|\nabla u(x_\eps)| \geq a-\eps.$$ Let $d$ be the distance from $x_\eps$ to $\Gamma(u)$ and consider the rescalings
$$u_\eps(x):= \frac{u(x_\eps+d x)}{d}.$$
These rescalings will still satisfy  \eqref{fb} say in $B_2,$ with $B_1\subset B_2^+(u_\eps)$ being tangent to $\Gamma(u_\eps)$. In fact, after a rotation we can also assume that
$$\nabla u_\eps(0)= t_\eps e_n, \quad a \geq t_\eps \geq a-\eps.$$
Since the $u_\eps$ are uniformly Lipschitz (up to extracting a subsequence) we can conclude that
$$u_\eps \to \bar u$$
uniformly on compacts in $B_2$. In fact, since the $u_\eps$ are uniformly $C^{2,\alpha}$ in the interior, the convergence is in the $C^{2,\alpha}$ norm on compact subsets of $B^+_2(\bar u) \cup B^-_2(\bar u).$

In particular, $\bar u$ solves \eqref{fb} and it satisfies
$$\bar u_n(0)=a, \quad |\nabla \bar u| \leq a \quad \text{in $B_2^+(\bar u)$},$$
with $B_1 \subset B_2^+(\bar u)$ tangent to $\Gamma(\bar u).$ In fact $\bar u$ solves $\mathcal F(D^2 \bar u)=0$ in $B_1$ and $\bar u \geq 0$ in $B_1$ but it is not identically zero because $\bar u_n(0)=a>0$. Thus $\bar u>0$ in $B_1.$ Moreover, since all the $u_\eps$ vanish at a point on $\p B_1$ it follows that $B_1$ is tangent to $\Gamma(\bar u).$

Now by Proposition 5.5 in \cite{CC}, $\bar u_n \in \mathcal S(\lambda,\Lambda)$, hence by the strong maximum principle
$$\bar u_n \equiv a \quad \text{in $\mathcal O$}$$
where $\mathcal O$ is the connected component on $B_2^+(\bar u)$ containing $B_1.$
On the other hand,
$$|\nabla \bar u| \leq a \quad \text{in $\mathcal O,$}$$ hence  
$$\bar u = a x_n + constant \quad \text{in $\mathcal O.$}$$
Moreover $\bar u =0$ at the point where the unit ball is tangent to $F(\bar u)$ thus
\be\label{ubar}\bar u = a (x_n +1)  \quad \text{in $B_2 \cap \{x_n >-1\}$}.\ee
Now, call (see \eqref{ab} and recall $a,b>0$)$$b_0=G^{-1}(a)>b.$$ Since $|\nabla \bar u| \leq b$ in $B_2^-(\bar u)$ and \eqref{ubar} holds, the two-plane solution $$p:=a(x_n+1)^+-b_0(x_n+1)^-$$ touches $\bar u$ by below on $B_2 \cap \{x_n=-1\}$. This is possible only if $p \equiv \bar u$ and this contradicts that $|\nabla \bar u| \leq b$ in $B_2^-(\bar u).$

Similarly, if we assume that $a< G(b)$ we can argue as above, starting with a point $x_\eps \in \{u < 0\}$ where $|\nabla u (x_\eps)| > b-\eps$ and reach a contradiction. Thus, we have shown that $$a=G(b)$$ and moreover
$$\bar u= a(x_n+1)^+-b(x_n+1)^-,$$  which proves the second part of the lemma.
\end{proof}

Since $\mathcal F$ is homogeneous of degree 1, after multiplication by a constant we can assume without loss of generality that $$a=b=1, \quad G(1)=1, \quad |\nabla u| \le 1,$$ 
and this assumption will be made throughout the paper from now on. Then, the following corollary holds.

\begin{cor}\label{eeta} Assume $0 \in \Gamma(u)$ and $t e_n \in B_2^+(u)$ for all $t \in (0,1]$. If
$$u_n(t e_n) \geq 1-\eta \quad \quad \quad \forall t \in (0,1],$$ for some $\eta >0$, then
$$|u-x_n| \leq \eps(\eta) \quad \text{in $B_2,$}$$ with $\eps(\eta) \to 0$ as $\eta \to 0.$
\end{cor}
\begin{proof} This follows immediately by compactness from Lemma \ref{uxn}. 

Assume by contradiction that there exist a sequence $\eta_k \to 0$ and a sequence of equi-Lipschitz solutions $u_k$ to a sequence of problems with operators $\mathcal F_k \in \mathcal E(\lambda, \Lambda)$ and such that 
\be\label{close}(u_k)_n(te_n) \geq 1-\eta_k \quad \text{for all $t \in (0,1]$, $t e_n \in B_2^+(u_k),$}\ee
but
\be\label{far}|u_k - x_n| > \delta \quad \text{at some point in $B_2$, for some fixed $\delta>0.$}\ee  
Then, from $u_k(0)=0$ and \eqref{close} we conclude that
$$u_k(e_n) \geq 1-\eta_k.$$ Since $|\nabla u_k| \leq 1$ this implies that
$$u_k >0 \quad \text{in $B_{1-\eta_k}(e_n)$}.$$

We can now argue as in the previous lemma and extract a subsequence which will converge uniformly on $\overline {B}_2$ to $\bar u=x_n$ (since $a=b=1$) and contradict \eqref{far}.

\end{proof}

We are now ready to prove Theorem \ref{T2}, assuming that Proposition \ref{Nonlinear_intro} holds.

First, we remark that if alternative $(i)$ of Proposition \ref{Nonlinear_intro} is satisfied, then we can rescale and iterate one more time. Precisely,  if $u$ satisfies alternative $(i)$, then the rescaling 
$$u_r(x):= \frac{u(rx)}{r}, \quad \quad \quad r=r_0,$$
will satisfy
$$|u_r - P_{r\bar M, \nu}|\leq \eps_r:= \eps r^{1+\alpha_0}, \quad \text{in $B_1$}.$$
In order to apply Proposition \ref{Nonlinear_intro} once again, we need to check that
$$\|r\bar M\|\leq \delta_0 \eps_r^{1/2}, \quad \eps_r \leq \eps \leq \eps_0.$$
This is clearly satisfied since (for $\eps_0$ small enough)
$$\|r \bar M\| \leq r \delta_0 \eps^{1/2} + C r \eps  \leq \delta_0 \eps_r^{1/2}= \delta_0 (\eps r^{1+\alpha_0})^{1/2}.$$

\medskip

{\it Proof of Theorem $\ref{T2}.$}
Let $\eps_0$ be a sufficiently small universal constant so that the conclusions of Proposition \ref{Nonlinear_intro}, Proposition \ref{c1} and Corollary \ref{maincor} hold for all $\eps \le \eps_0$, and let $\eps(\eta)$ be as in Corollary \ref{eeta}. 

According to Lemma \ref{uxn}, after a translation, a rotation and a rescaling, we can assume that
\be  \mbox{$0 \in \Gamma(u)$ and} \quad \quad |u - x_n| \leq \eps', \quad \text{in $B_1$,}\ee
for $\eps'$ small to be made precise later. If $\eps' \leq \eps_0$, then $\Gamma(u)$ is locally $C^{1,\alpha}$ and $u$ is a classical solution by Corollary \ref{maincor}.

By iterating Proposition \ref{c1} (say with $M \equiv 0$) and by interior $C^{1,\alpha}$ estimates for fully nonlinear equations it follows that,
\be\label{eps''} u_n(t e_n) \geq 1 -C \eps' \quad \quad \quad \forall t \in (0, 1/2], \ee
and
\be\label{eps'} u_n^+(0):= \lim_{t \to 0^+} \frac{u(t e_n)}{t} > 1- C  \eps',\ee
for some $C$ universal.

 We choose  $\bar \eta$ universal, small enough so that $$\eps(\bar \eta) \leq\eps_0,$$
and we claim that

\begin{lem}\label{bo}
\be\label{unen} u_n(t e_n) \geq 1-\bar \eta, \quad \text{for all $t > 0$.} 
\ee 
provided that $\eps'$ is chosen sufficiently small.
\end{lem}

Then, according to Corollary \ref{eeta},
\be 
|u- x_n| \leq \eps(\bar \eta) R \le \eps_0 R, \quad \quad \mbox{in $B_R,$}
\ee
for all $R$ large. This, combined with the improvement of flatness Proposition \ref{c1} (say with $M \equiv0$) implies that $u$ must be a two-plane solution and concludes the proof of our main Theorem \ref{T2}.

\medskip

We are left with the proof of Lemma \ref{bo}.

\medskip

{\it Proof of Lemma $\ref{bo}$.} Denote $C \eps'$ from \eqref{eps''} by $\eps'':=C \eps'$ and we will choose $\eps''$ (and therefore $\eps'$) later, so that $\eps'' \ll \bar \eta $.

Let $\bar t$ be the first $t$ for which \eqref{unen} fails, and assume without loss of generality after a rescaling that $\bar t=1$, i.e.
\be\label{equal}u_n(e_n)=1- \bar \eta, \quad u_n(t e_n) \geq 1-\bar \eta \quad \forall t \in [0,1].\ee 
Notice that after rescaling \eqref{eps'} is still satisfied, i.e.
$$u_n^+(0)> 1-\eps''.$$

From Corollary \ref{eeta} and \eqref{equal}, we conclude that
\be\label{starting}
|u-x_n| \leq \eps(\bar \eta) \leq \eps_0 \quad \text{in $B_2$.}
\ee
Corollary \ref{maincor} gives that $\Gamma(u) \cap B_1$ is a $C^{1,\alpha}$ graph in the $e_n$ direction with small norm.

Since $1-u_n \geq 0$ belongs to the class of solutions to linear equations with measurable coefficients $\mathcal S(\lambda,\Lambda)$ in $\{u>0\}$, the Harnack inequality and $(1-u_n)(e_n) = \bar \eta$ give that 
$$1-u_n \geq c \bar \eta \quad \text{in $B_{1/4}(\frac 12e_n) \subset B_1^+(u)$}.$$
In $B_1^+(u) \setminus B_{1/4}(e_n/2)$ we compare $1-u_n$ with the solution to $\mathcal M^+(D^2 w)=0$ which equals zero on $\p B_1^+(u)$ and $c \bar \eta$ on $\p B_{1/4}(\frac 12 e_n)$. Using the Hopf lemma in $C^{1,\alpha}$ domains together with $C^{1,\alpha}$ estimates up to the boundary we conclude that $w$ grows linearly away from $\Gamma(u)$ hence
\be u_n(te_n) \leq 1- c \bar \eta t, \quad t \in [0,1/2].
\ee
Integrating in the $e_n$ direction and using $u(0)=0$, and $c$, $\bar \eta$ are universal we find
\be\label{help}
u(te_n) \leq t - c_1 t^2, \quad \quad \forall t \in [0,1/2].
\ee
for some small universal constant $c_1>0$.

From \eqref{starting} we have
$$|u-x_n| \le \eps_0 \quad \quad \mbox{in $B_1$}.$$
We now apply Proposition \ref{Nonlinear_intro} and conclude that either alternative $(i)$ or $(ii)$ is satisfied. However if $\eps''$ is small enough, then alternative $(ii)$ cannot hold since otherwise by \eqref{eps'} 
$$1- c_0 \eps_0 \geq |\nabla u^+(0)| \geq u_n^+(0) \geq 1- \eps'',$$
and we reach a contradiction.
Similarly, after applying the conclusion of Proposition \ref{Nonlinear_intro} a number of $N$ times, we obtain that if $\eps''$ is sufficiently small depending on $N$ and $\eps_0$, $c_0$, $r_0$, then only alternative $(i)$ can hold for the $N$-iterations and conclude that 
\be \label{uW}|u-P_{M, \nu}| \leq \eps_0 r^{2+\alpha_0}, \quad \text{in $B_r$}, \quad r:=r_0^N,\ee
with \be\label{finalnu}
\left|\frac{\nabla u(0)}{|\nabla u (0)|} - \nu\right|\leq C \eps_0 r^{1+\alpha_0}, \quad \text{and} \quad M\nu=0.
\ee
Here we also used Corollary \ref{maincor} (see \eqref{cor2}).
Now notice that,
$$1-\eps'' \leq e_n \cdot \nabla u^+(0) \leq |\nabla u^+(0)| \leq 1,$$
which implies
\be
\left|\frac{\nabla u(0)}{|\nabla u (0)|} - e_n\right |\leq (2\eps'')^{1/2}.
\ee
Thus, if $\eps''$ is small enough depending on $r$, the inequality above together with \eqref{finalnu} gives that
\be |\nu - e_n| \leq 2 C \eps_0 r^{1+\alpha_0}.\ee
Thus, since $M\nu=0$ and $\|M\| \leq 1,$
$$\left |P_{M,\nu}(\frac r 2 e_n) - \frac r 2\right |=\left |\frac r 2 e_n \cdot (\nu-e_n)+ \frac{r^2}{8}(e_n-\nu)^T M (e_n-\nu)\right | \leq C' r^{2+\alpha_0} $$ which combined with \eqref{uW} gives that
$$\left |u(\frac r 2 e_n)-\frac r 2 \right | \leq 2C' r^{2+\alpha_0}.$$
This contradicts \eqref{help} for $t=r/2$, as long as $r$ is small enough universal, i.e. $N$ is large enough and $\eps''$ is small enough.
\qed

\section{The transmission problem}

In this section we study properties of solutions to the nonlinear transmission problem \eqref{tra} below. This type of transmission problem appears as the linearization to the free boundary problem \eqref{fb} and our goal is to obtain a version of the key Proposition \ref{Nonlinear_intro} in this linearized setting. In Section 5 we will use compactness methods and extend the result to the nonlinear setting. 

Consider the transmission problem
\be\label{tra}
\begin{cases}
\mathcal F(D^2 v) =0 \quad \text{in $B_1 \cap \{x_n \neq 0\},$}\\
v_n^+= b v_n^- \quad \text{on $B_1 \cap \{x_n = 0\}$,}\\
\end{cases}
\ee
with $0< b_0\leq  b \leq b_1$ and say $\|v\|_\infty \leq 1.$ Here $\mathcal F \in \mathcal E(\lambda, \Lambda)$.

By abuse of notation, in this section we denote by $v^\pm$ the restrictions of $v$ to $B_1 \cap \{x_n \geq 0\}$ and respectively $B_1 \cap \{x_n \leq0\}.$
Call,
$$\mathcal C^{1,\alpha}(B_1^\pm)= C^{1,\alpha}(B_1 \cap \{x_n\geq 0\}) \cap C^{1,\alpha}(B_1 \cap \{x_n\leq 0\})$$

In \cite{DFS} the authors showed that solutions to the transmission problem above belong to the class $\mathcal C^{1,\alpha}(B_1^\pm)$, for some universal $\alpha$. We recall here the definition of viscosity solutions to the problem \eqref{tra} and the precise result from \cite{DFS}. Constants depending on $n, \lambda, \Lambda, b_0, b_1$ are called universal.

\begin{defn}\label{def_visc}
We say that $v \in C(B_1)$ is a viscosity subsolution (resp. supersolution) to \eqref{tra} if

(i)   $ \mathcal{F} (D^2 v^\pm) \geq 0  \ (\text{resp. $ \leq 0$})\quad \text{in $B^\pm_1$},$ in the viscosity sense;

(ii) $v_n^+ \ge b \, v_n^-$ on $B_1 \cap \{x_n = 0\}$ in the viscosity sense which means that if 
$$P(x')+px_n^+ - qx_n^-$$ touches $v$ by above (resp. by below) at $x_0 \in \{x_n=0\}$ for some quadratic polynomial $P(x')$ in $x'$, then $$ p - b \, \, q \geq 0 \quad (\text{resp.$ \leq 0$}).$$

\end{defn}

If $v$ is both a viscosity subsolution and supersolution to \eqref{tra}, we say that $v$ is a viscosity solution to \eqref{tra}.


The next result is contained in Theorems 3.2 and 3.3 in \cite{DFS}.

\begin{thm}[DFS]\label{lineareg}Let $v$ be a viscosity solution to \eqref{tra} such that $\|v\|_\infty \leq 1$. Then $u\in \mathcal C^{1,\alpha}(B_{1/2}^\pm)$ with a universal bound on the norms. In particular, there exists a universal constant $C$ such that
\be\label{lr}|v(x) - v(0) -(\nabla_{x'}v(0)\cdot x' + px_n^+ -  qx_n^-)| \leq C r^{1+\alpha}, \quad \text{in $B_r$}\ee for all $r \leq 1/4$ and with
\begin{equation}\label{idapbq=0}   p - b \, \,  q=0.\end{equation}
\end{thm}

 We also recall the definitions from \cite{DFS} of the general classes of subsolutions $\underline{\mathcal S^*}$, supersolutions $\overline{\mathcal S^*}$ and solutions ${\mathcal S^*}$ suited for the transmission problem \eqref{tra}.
 In what follows,
$$L:=\{x_n=0\}.$$
We denote by $\underline{\mathcal{S^*}}(\lambda,\Lambda)$ the class of functions $w \in C(B_1)$ such that   $$\mathcal M^+_{\lambda,\Lambda}(D^2w) \geq 0 \quad \text{in $B_1 \setminus L$}, \quad \quad\mbox{and} \quad w_n^+ - b w_n^- \geq 0, \quad  \hbox{on $B_1 \cap L.$}$$
Analogously, $\overline{\mathcal{S^*}}(\lambda,\Lambda)$ denotes the class of functions 
$w \in C(B_1)$ such that $$-w \in \underline{\mathcal{S^*}}(\lambda,\Lambda),$$
and $$\mathcal{S^*}(\lambda,\Lambda):=\underline{\mathcal{S^*}}(\lambda,\Lambda)\cap \overline{\mathcal{S^*}}(\lambda,\Lambda)$$
For simplicity of notation we will drop the dependence on $\lambda, \Lambda$. 

We restrict our attention to the case of concave operators.

\begin{lem}\label{vtt}
Let $v$ be a viscosity solution of \eqref{tra} and assume $\mathcal F$ is concave. 
Then the tangential second order quotients $v^h_{\tau\tau}$ are subsolutions, i.e.
$$ v^h_{\tau\tau}(x):= \frac{ v(x+h\tau)+  v(x-h\tau) -2  v(x)}{h^2}  \quad \in \underline{\mathcal {S^*}}.$$
Here $\tau$ is a unit direction with $\tau \perp e_n$, and $h>0$.
\end{lem}

\begin{proof}
The fact that $v^h_{\tau\tau}$ is a subsolution in $B_1 \setminus L$ is standard. By Theorem \ref{lineareg} $v^h_{\tau\tau} \in \mathcal C^{1,\alpha}(B_1^\pm)$, hence the transmission condition on $L$ is satisfied in the classical sense and therefore it holds also in the viscosity sense.
\end{proof}

In Lemma 3.5 in \cite{DFS} it was shown by the use of an explicit barrier that the Harnack inequality for the class $\mathcal S^*$ follows from the standard Harnack inequality.
  
Next we state the weak Harnack inequality for the classes $\overline{\mathcal S^*}$ and $\underline{\mathcal S^*}$.

\begin{lem} Let $w \in \overline{\mathcal S^*}$ in $B_1$, $w \geq 0$ and $inf_{B_{1/2}} w \leq 1.$ There exist universal constants $0<\mu<1$ and $M>1$ such that 
$$|\{w \leq M\} \cap B_{1/2} | \leq \mu.$$
\end{lem}

\begin{proof}
The proof follows from the standard weak Harnack inequality once we know that $w \le C'$ at some point in the set $\{|x'| \le 3/4, \quad |x_n|\ge \delta \} \cap B_{3/4}$. Otherwise we can use comparison principle in the cylinder $\mathcal C_0:=\{|x'| \le 3/4, \quad |x_n|\le \delta \}$ and conclude that 
$$w \ge C' \left((5/8)^2-|x'|^2 + n (\Lambda /\lambda) x_n^2 \right) $$  
provided that $\delta$ is small. This implies that $w>1$ in $B_{1/2} \cap \mathcal C_0$ and we contradict the hypothesis that  $inf_{B_{1/2}} w \leq 1.$
\end{proof}

As in the proof of Theorem 4.8 in \cite{CC}, we can iterate the lemma above and the standard weak Harnack inequality at smaller scales and obtain the version of weak Harnack inequality for subsolutions (here $w_+$ denotes the positive part of $w$.)

\begin{thm}[Weak Harnack]\label{TWH} Let $w \in \underline{\mathcal{S^*}}$ in $B_1$. Then, for any $p>0,$
\be\label{WHA}\|w_+\|_{L^\infty(B_{1/2})} \leq C(p)\, \, \|w_+\|_{L^p(B_{3/4})},\ee
with $C(p)$ depending on $p$ and the universal constants.
\end{thm}

As mentioned in the Introduction, the $C^{2,\alpha}$ estimates play an important role in our analysis, however the Evans-Krylov theorem for the transmission problem \eqref{tra} is not known. In the next proposition we show that, if in addition the solution $v$ is monotone decreasing in the $e_n$ direction, then either there is indeed a pointwise $C^{2,\alpha}$ type estimate at $0$ or the $e_n$ derivative of $v$ is strictly negative.  

We assume that
$$\mbox{$v$  solves \eqref{tra},} \quad \quad  \|v\|_\infty \leq 1, \quad \quad \mbox{$\mathcal F$ is concave}.$$ 
We wish to prove the following Harnack type inequality for $v_n^\pm.$

\begin{prop}[Linear dichotomy]\label{dic} Let $v$ be as above and suppose that
\be\label{monotone1} v_n^\pm \leq 0.\ee
Then, there exist universal constants $c_0, r_0>0$ such that either 
\begin{enumerate}
\item \be\label{1}v^+_n(0) \leq -c_0\ee 
or \item  \be \label{2} |v(x)- Q(x')| \leq \frac 1 4 r_0^2, \quad \text{in $B_{r_0}$}\ee
 \end{enumerate}
 where $Q(x')$ is a quadratic polynomial in the $x'$ direction with $\|Q\| \leq C$ universal and
 \be\label{fq1} \mathcal F(D^2Q)=0.\ee
\end{prop}

\begin{proof} Let $r_0$ be given, to be specified later.
Now, assume by contradiction that we can find a sequence $c_k \to 0$ and sequences of convex operators $\mathcal F_k \in \mathcal E(\lambda, \Lambda),$ constants $b_0 \leq b_k \leq b_1$ and bounded monotone solutions $v_k$ to \eqref{tra} which satisfy $$\partial_n v^+_k(0) > -c_k$$ and for which (ii) does not hold. Then, up to extracting a subsequence, $v_k$ converges uniformly on compacts, and in the $C^{1,\alpha}$-norm from either side of $L$,  to  a solution $\bar v$ for a limiting problem 
  \be\label{tra_limit}
\begin{cases}
\mathcal {\bar F}(D^2 \bar v) =0 \quad \text{in $B_1 \setminus L,$}\\
\bar v_n^+= \bar b \bar v_n^- \quad \text{on $B_1 \cap L$,}\\
\end{cases}
\ee
with $$\|\bar v\|_\infty \leq 1$$and 
\be\label{con} \bar v^\pm_n \leq 0, \quad \bar v_n^+(0)=0.\ee
Also, $\mathcal {\bar F} \in \mathcal E (\lambda, \Lambda)$, $\mathcal{ \bar F}(0)=0$ and $\mathcal {\bar F}$ is concave.

After subtracting a linear function in the $x'$ variable we can assume that (in view of the free boundary condition $\bar v_n^-(0)=0,$)
\be\label{wlg} \bar v(0)=0, \quad \nabla \bar v(0)=0.\ee

\smallskip

{\bf Step 1.} We prove that $$D^2_{x'} \bar v \le C I_{x'} \quad \mbox{in $B_{1/2}$,}$$ or in other words that $\bar v$ is uniformly semiconcave in the $x'$ variable.

For this it suffices to show that 
\be\label{step1}\bar v^h_{\tau\tau}(x) \leq C \quad \text{in $B_{1/2}$}\ee
independently of $h$, where $\bar v_{\tau \tau}^h$ are the tangent second difference quotient as in Lemma \ref{vtt}.
We claim that
\be\label{bound1} \bar v^h_{\tau \tau} \leq \frac{C}{x^2_n}, \quad \text{in $B_{3/4}$}\ee
which combined with Lemma \ref{vtt} and the weak Harnack inequality \eqref{WHA} for small $p>0$ implies the desired bound.
To prove \eqref{bound1}, notice that by the Evans-Krylov $C^{2,\alpha}$ interior estimates in $B_{|x_n|/2}(x)$ we have
$$|D^2 \bar v(x)| \leq \frac{C}{x_n^2}\|\bar v\|_{L^\infty} \le \frac{C}{x_n^2}, \quad \text{in $B_{3/4} \setminus L$}.$$
Since we can write,
$$\bar v^h_{\tau \tau}(x_0)= \int_{-1}^1 \bar v_{\tau \tau}(x_0+th\tau)(1-|t|)dt,$$ \eqref{bound1} follows.
\end{proof} 

{\bf Step 2}.
In this step we wish to show that 
\be\label{quadratic}
\|\bar v\|_{L^\infty(B_r)} \leq C r^2, \quad r \leq 1/4,
\ee with $C$ universal. Below, the constant $C$ may change from line to line.

Set
\be\label{tilde} \tilde v(x) = \frac{\bar v (rx)}{r^2}, \quad x \in B_1.
\ee

Then $\tilde v$ satisfies \eqref{tra_limit}-\eqref{con}. From \eqref{wlg} we also have, 
$$\tilde v(0)=0, \quad \nabla \tilde v(0) =0.$$
Then the conclusion of Step 1 implies that
\be\label{bound}
\tilde v(x', 0) \leq C |x'|^2, \quad \text{in $B_1 \cap L$}.
\ee
Hence, by the monotonicity of $\tilde v$,
\be\label{boundinside} \tilde v \leq C \quad \text{in $B_1^+:= B_1 \cap \{x_n>0\}$.}\ee
Next we claim that
\be\label{vmc}
\tilde v \geq -M C  \quad \text{in $B_{1/3}^+(\frac 12 e_n)$,}
\ee 
for some large constant $M$ to be specified later.
Suppose by contradiction that this does not hold. Then by Harnack inequality for $C - \tilde v \ge 0$ in $B_1^+$,
$$C-\tilde v \geq (1+M)C d^K \quad \text{on $x_n=d, |x'| \leq 1/3$}$$
with $K$ universal. In particular, if $$d^K \geq \frac{1}{1+M}$$
then,
\be\label{xn=d} \tilde v \leq -C \quad \text{on $x_n=d, |x'| \leq 1/3$}.\ee
We compare $\tilde u$ with the explicit barrier:
$$\phi(x):= 10C |x'|^2 - A x_n^2 -x_n,$$
where $A=A(C,\lambda,\Lambda)$ is chosen so that
$$\mathcal {\bar F}(D^2 \phi) \leq \mathcal M^+(D^2 \phi) \leq 0.$$ Call
$$R:=\{0 <x_n<d, |x'|< 1/3\}.$$
We show that for $d$ small enough (hence $M$ large enough) \be\label{compare}\tilde v \leq \phi \quad \text{on $\p R$}.\ee
Thus, we conclude that the inequality holds in $R$ and we reach a contradiction because,
$$0=\tilde v_n(0) \le \phi_n(0)=-1.$$ 
Now we check \eqref{compare}. On $\{x_n=d\}$
this follows from \eqref{xn=d}, if $d$ is chosen small depending on $C$ and $A.$ Similarly, on $\{x_n=0\}$ the desired bound follows immediately from \eqref{bound}.
Finally in the set $\{ 0<x_n<d, |x'|=1/3\}$ we use \eqref{boundinside} and again we obtain \eqref{compare} for $d$ sufficiently small. In conclusion the claim \eqref{vmc} holds.

Using the $\tilde v$ in decreasing in the $e_n$ we obtain (after relabeling $C$)
$$\tilde v \geq -C \quad \text{in} \quad  B_{1/3}$$
and recall that $\tilde v (0)=0.$ We apply Harnack inequality for $ \tilde v + C \in \mathcal S^*$ (see Lemma 3.4 in \cite{DFS}) and conclude that $$|\tilde v| \leq C \quad \text{in $B_{1/4}$},$$
which after rescaling gives the desired claim \eqref{quadratic}.

\smallskip

{\bf Step 3.} We prove that if $\bar v$ solves \eqref{tra_limit} and satisfies \eqref{con}-\eqref{wlg}-\eqref{quadratic}, then there exists a universal $r_0$ such that alternative (ii) holds for $\bar v$ with right hand side $\frac 1 8 r_0^2$ instead of $\frac 14 r_0^2$.

This claim follows by compactness from the classification of global solution obtained in Lemma \ref{global_lemma} below.

 Assume by contradiction that there exists a sequence of $\delta_k \to 0$ and of solutions $\bar v_k$ to a sequence of problems $\eqref{tra_limit}_k$ (satisfying the same properties as $\bar v$) for which the alternative (ii) fails in the ball of radius $\delta_k.$ Denote the quadratic rescalings by
$$w_k(x)= \frac{\bar v_k(\delta_k x)}{\delta_k^2}.$$ Then, up to extracting a subsequence, $w_k$ converges uniformly on compacts to a  global solution $U$ to a limiting transmission problem ($\mathcal G \in \mathcal E(\lambda, \Lambda), b_0 \leq g \leq b_1,$)
  \be\label{global_1}
\begin{cases}
\mathcal {G}(D^2 U) =0 \quad \text{in $\R^n \setminus L,$}\\
U_n^+= g U_n^- \quad \text{on $L$,}\\
\end{cases}
\ee with the convergence being in the $C^{1,\alpha}$ norm from either side of $L$, up to $L$, and in the $C^{2,\alpha}$ norm in the interior. Clearly the global solution $U$ also satisfies \eqref{con}-\eqref{wlg}-\eqref{quadratic}, and the operator $\mathcal G$ is concave as the limit of the corresponding $\mathcal {\bar F}_k$.

Then according to Lemma \ref{global_lemma} below, we conclude that $U=Q(x')$ with $Q(x')$ a pure quadratic polynomial in the $x'$-direction and with $\mathcal G(D^2 Q)=0$. Therefore, for $k$ large, $\bar v_k$ satisfies the alternative (ii) in $B_{\delta_k}$ with $Q_k:= Q + t_k |x'|^2$ for an appropriate choice of $t_k$'s, so that $\mathcal {\bar F}_k(D^2 Q_k)=0$ and $t_k \to 0$ as $k \to \infty$. We reached a contradiction and therefore we have established Step 3.

\smallskip

{\it End of the proof:} By Step 3, $\bar v$ satisfies  alternative (ii) with right hand side $\frac 1 8 r_0^2$ for some quadratic polynomial $\bar Q(x')$ with $\mathcal {\bar F}(D^2 \bar Q)=0$. As above, this means that the $v_k$'s satisfy the alternative (ii) for all $k$ large, which is a contradiction.

\qed


\begin{lem}\label{global_lemma} Let $U$ be a global solution to  \be\label{global}
\begin{cases}
\mathcal {G}(D^2 U) =0 \quad \text{in $\R^n \setminus L,$}\\
U_n^+= g U_n^- \quad \text{on $L$,}\\
\end{cases}
\ee
with $\mathcal G \in \mathcal E(\lambda, \Lambda)$,  $\mathcal G$ concave, $0<b_0 \leq g \leq b_1$, and
\be\label{con_glob} U^\pm_n \leq 0, \quad  |U(x)| \leq C |x|^2,\ee
for some constant $C$.
Then $U=Q(x')$ with $Q(x')$ a pure quadratic polynomial in the $x'$-direction, such that $\mathcal{G}(D^2 Q)=0.$
\end{lem}

\begin{proof}  
Let $$U_k(x):=\frac{U(r_k x)}{r_k^2}  \quad \quad \mbox{with} \quad r_k \to \infty,$$ be a sequence of blow-downs which converges uniformly on compacts to another global solution $\bar U$ to \eqref{global}-\eqref{con_glob}.

Let $\tau$ be a unit tangential direction, $\tau \perp e_n$ and denote by,
$$\gamma=\gamma(\tau):= \sup_{\R^n \setminus L} \bar {U}_{\tau \tau}(x).$$
Notice that $\alpha \ne \infty$. Indeed, from Step 1 in the proof of Proposition \ref{dic} we obtain that $\partial_{\tau \tau}U^h_k$ is bounded above in $B_1$ independent of $h$, and this implies that $\gamma$ is well defined.

We claim that 
\be\label{utt}
\bar U_{\tau\tau} \equiv \gamma \quad \text{on $\{x_n \neq 0\}$}.\ee
Let $x_0$ be a point, say for simplicity in $B_{1} \setminus L$, and let us show that $$\bar U_{\tau \tau}(x_0)=\gamma.$$
Since $U_k \to \bar U$ in $C^{2,\alpha}$ in a small ball $B \subset \{x_n \neq 0\}$ around $x_0$,
clearly,
$$\bar U_{\tau \tau}(x_0) \leq \gamma.$$

Assume by contradiction that the inequality above is strict. Then for all $k$ large, 
$$(U_k)_{\tau \tau}(x_0)\leq \gamma -\delta,$$
for some small $\delta>0.$
By $C^{2,\alpha}$ regularity,
$$(U_k)_{\tau \tau}\leq \gamma -\frac\delta 2, \quad \text{in $B_{c}(x_0)$} \subset B.$$
hence, for all $h>0$ small
$$(U^h_k)_{\tau \tau}\leq \gamma -\frac\delta 2, \quad \text{in $B_{c'}(x_0)$}.$$
On the other hand form the definition of $\gamma$ we have
$$(U^h_k)_{\tau \tau} \leq \gamma, \quad \text{in $B_{3/4}$}.$$
Since $(U^h_k)_{\tau \tau} \in \underline {\mathcal S ^*}$ we can construct an explicit upper barrier in $B_{3/4} \setminus B_{c'}(x_0) $ (see Lemma 3.4 in \cite{DFS}), and conclude that 
$$(U^h_k)_{\tau \tau} \leq \gamma - c(\delta), \quad \text{in $B_{1/2}$}.$$
for all $k$ large.
This contradicts the definition of $\gamma,$ and the claim \eqref{utt} is proved.
\smallskip

Next we show that $\bar U$ equals a quadratic polynomial $Q^+$ (resp. $Q^-$) in $\{x_n>0\}$ (resp. $\{x_n<0\}$) and
\be\label{P} Q^+=Q^- \quad \text{on $L$}, \quad Q^\pm(0)=0,
\quad  \nabla Q^\pm(0)=0, \quad Q^\pm_n \leq 0.\ee

We know that 
$$\bar U_{\tau\tau}= constant, \quad \forall \tau \perp e_n,$$
hence
$$\bar U(x)=b(x_n) + x' \cdot a(x_n) + Q(x'), \quad a(x_n)=(a_1(x_n), \ldots, a_{n-1}(x_n)),$$
for some quadratic polynomial $Q$ and functions $b$, $a_i$ depending on one variable. 

In the set $\{x_n>0\}$ we have for $i<n$
$$\bar U_i \in \mathcal S \quad \Longrightarrow \quad a_i(x_n) + Q_i(x') \in \mathcal S \quad \Longrightarrow \quad a_i(x_n) \in \mathcal S,$$
where we have used that $Q_i$ is a linear function. This means that $a_i$ is linear in the set $\{x_n>0\}$. 
Now we use $\bar U_n \in \mathcal S$ and we argue as above to find that $b'$ is linear in $\{x_n>0\}$ and our claim \eqref{P} is proved.

Moreover, $$Q^\pm_{ni}=0 \quad i\neq n, \quad Q^+_{ij}=Q^-_{ij}, \quad i,j\neq n.$$
Finally, since $Q^\pm$ both solve (from either side of $L$), $$\mathcal{G}(D^2Q^\pm)=0$$ we conclude that
$$Q^+_{nn}=Q^-_{nn}$$
and therefore 
$$Q^+ \equiv Q^-.$$ Let us call this common polynomial $Q.$
In particular, since $Q_n \leq 0$ it follows that 
$Q_n\equiv 0$ . Thus, $Q$ is a pure quadratic polynomial in the $x'$ direction.

\smallskip

Finally, we need to show that $U \equiv Q.$

First notice that, for all tangential directions $\tau$,
$$\sup U_{\tau \tau} = Q_{\tau \tau}.$$
Thus, $U-Q$ is concave in the $x'$-direction. Moreover,
\be\label{uq}
(U-Q)(0)=0, \quad \nabla (U-Q)(0) =0.\ee Hence,
$$U(x', 0) \leq Q(x')$$ and by the monotonicity of $U$ we get
$$U(x', x_n) \leq Q(x') \quad \text{on $x_n \geq 0$}.$$
Since $U-Q \in \mathcal S$, by \eqref{uq} and Hopf lemma we conclude that 
$U = Q$ in $\{x_n \ge 0\}$. By the monotonicity of $U$ we have $U \ge Q$ in the set $\{x_n \le 0\}$, hence $U^- = Q$ again by Hopf lemma. 
\end{proof}

\section{The free boundary problem}

In this section we prove Proposition \ref{c1} and Proposition \ref{Nonlinear_intro}.

Let $u$ be a Lipschitz continuous viscosity solution to \eqref{fb} in $\R^n$. Recall that,
following the arguments in Section 3, after an initial dilation we can assume without loss of generality that (see \eqref{ab})
$$a=b=1, \quad G(1)=1,$$ hence $|\nabla u|\leq 1.$

\subsection{Notation}
As in the Introduction, we denote by $P_{M,\nu}$ the quadratic polynomial 
$$P_{M,\nu}(x):=x \cdot \nu + \frac 1 2 x^T M x,$$
with  
$$|\nu|=1, \quad \quad \mbox{and} \quad  M \nu=0, \quad \mathcal F(M)=0,$$
and by $V_{a^\pm,M,\nu}$ the two-phase quadratic polynomial of slopes $a^+, a^- \in [\frac 12,2]$
$$V_{a^\pm, M, \nu}(x):= a^+P_{M,\nu}^+(x) - a^-P_{M,\nu}^-(x), \quad \quad a^+=G(a^-).$$
Given a continuous function $v$ in say $B_1$, we denote its $\eps$-linearization around the function $V_{a^\pm, M, \nu}$ above as
\be\label{utilde}
\tilde v_{a^\pm, M,\nu,\eps}(x):=\begin{cases}
 \dfrac{v(x) - a^+P_{M,\nu}(x)}{a^+ \eps}, \quad \text{if $x \in B_1^+(v) \cup \Gamma(v)$}\\
 \ \\
  \dfrac{v(x) - a^-P_{M,\nu}(x)}{a^- \eps}, \quad \text{if $x \in B_1^-(v).$}
\end{cases}
\ee

In what follows, we will typically drop the indices from $V, P, \tilde v$ whenever there is no possibility of confusion.

\begin{rem}\label{over} We remark that if $v,w \in C(B_1)$ and $v \geq w$ in $B_1$ then $\tilde v \geq \tilde w$ in $B_1.$ This claim is obvious if $v, w$ have the same sign. If $v \geq 0 > w$, then
$$\tilde v := \frac{v}{a^+}- \frac{1}{\eps} P \geq - \frac{1}{\eps} P \geq  \frac{w}{a^-}- \frac{1}{\eps} P= \tilde w.$$
\end{rem}

\medskip

We now want to construct appropriate $\eps$-perturbations of $V$ which are strict subsolutions to the two-phase problem in say $B_1.$ Similarly, one can construct strict supersolutions.

Given $\eps, \delta>0,$ assume that
$$\|M\| \leq \delta \eps^{1/2}.$$
Let $N$ be a $n \times n$ diagonal matrix such that 
$$\mathcal M^-(N)>0,$$
and let $p,q, A \in \mathbb R$ and $\xi' \in \R^{n-1}.$ Given $\eps>0$, set 
$$\bar a^+:= a^+(1+\eps p), \quad \bar a^-= a^-(1+ \eps q ), \quad \bar M:= M+ 2\eps N$$
$$Q(x):= x \cdot e_n+ \frac 1 2 x^T\bar M x+ A\eps + \eps \xi'\cdot x'$$
and call,
\be\label{v1}v(x):= \bar a^+Q^+ -  \bar a^-Q^-.\ee

\begin{lem}\label{sub_1} There exists a constant $k$ depending on $p,q,N, \xi'$  and the universal parameters, such that if $p, q$ satisfy
\be\label{sub_2}a^+ p > a^-G'(a^-)q+ k\delta^2,\ee 
then  $v$ defined in \eqref{v1} is a strict subsolution to \eqref{fb} in $B_1$ for  all $\eps$ sufficiently small. \end{lem}
\begin{proof}
In $B_1^\pm(v)$ since $\mathcal F$ is homogeneous of degree 1 and $\mathcal F(M)=0$
$$
\mathcal{F}(D^2v) = \mathcal{F}(\bar a^\pm \bar M)= \bar a^\pm \mathcal F(M+2\eps N)\geq \bar a^\pm \mathcal F(M)+ 2\bar a^\pm \eps  \mathcal M^-(N) >0.$$
For the free boundary condition, let us compute on $\Gamma(v)$ (dependence of constants on the parameters is not explicitly noted, and constants may change from line without being renamed)
$$|\nabla v^+|= \bar a^+ \left(1+ |\bar M x|^2 + \eps^2 |\xi'|^2+ 4  \eps  (Nx)\cdot e_n +2 \eps (\bar Mx)\cdot \xi' \right)^{1/2}.$$ 
We use that $|\bar M| \le 2 \delta \eps^{1/2}$ and that $(Nx)\cdot e_n=O(\eps^{1/2})$ on $\Gamma(v)$ , hence
$$|\nabla v^+| \le \bar a^+(1-C \delta^2 \eps) \geq a^+(1+\eps p) - C \delta^2 \eps,$$
and similarly 
$$|\nabla v^-|\leq \bar a^-(1+C' \delta^2 \eps)= a^-(1+q\eps )+ C' \delta^2 \eps.$$
Thus, 
$$G(|\nabla v^-|) = G(a^-)+ q\eps  a^-G'(a^-)+ C'\delta^2 \eps +o(\eps)$$
$$< a^+(1+\eps p) - k\delta^2\eps+ C'\delta^2 \eps + o(\eps)<  a^+(1+\eps p) -C \delta^2 \eps \leq |\nabla v^+|$$
as long as $\eps$ is small enough (depending on $p,q,N,\xi', \delta$) and for the appropriate choice of $k$.
\end{proof}

\begin{rem}
In the inequality above it suffices only to assume that the modulus of continuity $\omega$ of $G'$ satisfies $\omega(0+) \le \delta^2$. 
\end{rem}
\begin{rem}\label{rem_w}
In particular, since the $v_t(x):=v(x+te_n)$ form a continuous family of subsolutions with $t \to -\infty$, we conclude that $u$ satisfies the comparison principle with translates of $v$. Hence, by Remark \ref{over}, $\tilde u$ and $\tilde v_t$ also satisfy the comparison principle. 
It is also easily seen that $\tilde v_{\eps t}$ converges locally uniformly as $\eps \to 0$ to the following function 
$$w(x):= t+ A+ \xi' \cdot x'+ px_n^+-qx_n^-+ x^TNx.$$ We will use this fact in the next subsection.

\end{rem}

\subsection{$C^{1,\alpha}$ estimates}

In this subsection, we prove Proposition \ref{c1}. Arguing as in \cite{DFS}, we first establish the following Harnack type inequality for $u$.

\begin{lem}\label{HI}There exist
universal constant $\bar \eps, \delta>0$ such that if $u$ satisfies \be\label{control}  u(x) \geq V(x):=V_{a^\pm, M, e_n}(x), \quad \text{in $B_1$}\ee with \be\label{nond}\|M\| \leq \delta \eps^{1/2}, \quad  0<\eps \leq \bar \eps,\ee  and at $\bar x=\dfrac{1}{5}e_n$ \be\label{u-p>ep2}
u(\bar x) \geq V(\bar x) + \eps, \ee then \be u(x) \geq
V(x) + c\eps \quad \text{in $\overline{B}_{1/2},$}\ee for some
$0<c<1$ universal.  \end{lem}

\begin{proof} 

The proof follows the lines of Lemma 4.3 in \cite{DFS}. For the reader convenience we sketch the main details.

We chose a specific function of the form $w$ from Remark \ref{rem_w}. Precisely, let
$$W(x):= \frac 1 8 + px_n^+-2C x_n^- - |x'|^2+ C x_n^2, \quad C=2\frac{\Lambda(n-1)}{\lambda}$$
that is $$A=0, \quad t=\frac 1 8, \quad \xi'=0, \quad q=2C, \quad N= diag\{-1,\ldots, -1, C\},  \quad \mathcal M^-(N)>0$$
and with
$$a^+p = a^- G'(a^-)q +1.$$
Call, for $\gamma>0,$
$$R_\gamma:=\{-1/2<x_n<\gamma, |x'|<1/2\}.$$
One can find $\gamma,\eta, c > 0$ small universal, so that
$$W \leq -c \quad \text{on $\p R_\gamma \setminus \{x_n=\gamma, |x'|\leq 1/2\}$}$$
and
$$W \geq c  \quad \text{on $B_\eta$}.$$
Now, let
$$w(x):= c_1W(x),$$
with $c_1$ to be specified later and call $v$ the function as in \eqref{v1}, associated to our choice of $A, \xi', p,q,N.$
In view of Proposition \ref{sub_1}, $v$ is a strict subsolution provided that $\delta$ is chosen sufficiently small. From Remark \ref{rem_w}, we conclude that $\tilde v_{\frac{c_1}{8}\eps}$ and $\tilde u$ satisfy the comparison principle, with 
$\tilde v_{\frac{c_1}{8}\eps}$ converging uniformly to $w$, as $\eps \to 0.$

In particular,  for $\eps$ small, 
\be\label{negative}\tilde v_{\frac{c_1}{8}\eps}<0 \quad \text{on $\p R_\gamma \setminus \{x_n=\gamma, |x'|\leq 1/2\}$}\ee
and
$$\tilde v_{\frac{c_1}{8}\eps} \geq \frac{c_1 c}{2}  \quad \text{on $B_\eta$}.$$
It follows  that if \be\label{boundary}\tilde u \geq \tilde v_{\frac{c_1}{8} \eps} \quad \text{on $\p R_\gamma$},\ee
then the inequality holds in $R_\alpha$ as well, and in particular
$$\tilde u \geq \frac{c_1 c}{2} \quad \text{on $B_\eta$}.$$
From this, the desired claim immediately follows in $B_\eta$. A standard covering argument gives the claim in the full $B_{1/2}.$

We are left with the proof of \eqref{boundary}. From assumption \eqref{control} and \eqref{negative} we have that 
\be\label{negative2}\tilde v_{\frac{c_1}{8}\eps}<0 \leq \tilde u \quad \text{on $\p R_\gamma \setminus \{x_n=\gamma, |x'|\leq 1/2\}.$}\ee

Moreover, by Harnack inequality and assumption \eqref{u-p>ep2}, one can guarantee that 
$$u(x)-V(x) \geq c_0 \eps \quad \text{on $ \{x_n=\gamma, |x'|\leq 1/2\},$}$$
for some $c_0(\alpha)$ universal. Hence, 
\be\label{almost} \tilde u \geq c_0 \geq 2 c_1\sup W   \quad \text{on $ \{x_n=\gamma, |x'|\leq 1/2\},$}\ee
if $c_1$ is chosen appropriately. Again, from the uniform convergence of $\tilde v_{\frac{c_1}{8} \eps}$ we obtain 
$$\tilde u \geq \tilde v_{\frac{c_1}{8}\eps}  \quad \text{on $ \{x_n=\gamma, |x'|\leq 1/2\},$}$$ and our claim is proved.

\end{proof}

{\it Proof of Proposition $\ref{c1}.$} The proof follows the line of Lemma 5.1 in \cite{DFS}. For the reader convenience we sketch the details. 

\smallskip

We divide the proof into 3 steps.

 \

\textbf{Step 1 -- Compactness.} Fix $r_1$ universal to be made precise later in Step 3. Assume by contradiction that we
can find sequences $\eps_k \rightarrow 0, \delta_k \to 0$ and a sequence $u_k$ of
solutions to \eqref{fb} in $B_1$ for a sequence of operators $\mathcal F_k \in \mathcal E(\lambda,\Lambda)$ and free boundary conditions $G_k$ such that
\begin{equation}\label{flat_k}|u_k -V_{a^\pm_k, M_k, e_n}| \leq \eps_k \quad \text{for $x \in B_1$,  $0 \in \Gamma(u_k),$}
\end{equation} with 
$$\|M_k\| \leq \delta_k \eps_k^{1/2}, \quad M_k e_n=0, \quad a^+_k=G_k(a^-_k), \quad a^\pm_k \in [1/2,2],$$
but $u_k$ does not satisfy the conclusion \eqref{flat2flat2} of the proposition.

Let $\tilde u_k := \tilde u_{a^\pm_k,M_k, e_n, \eps_k}$ be defined as in \eqref{utilde}.
Then \eqref{flat_k} gives,
\begin{equation}\label{flat_tilde}|\tilde{u}_{k}|\leq C
\quad \text{for $x \in B_1$}.
\end{equation}

From Lemma \ref{HI} we obtain that the oscillation of $\tilde u$ decreases by a factor $1-c$ as we restrict from $B_1$ to $B_{1/2}$. This result can be iterated $m$-times provided that $\eps_k (2(1-c))^m \le \bar \eps$. By Ascoli-Arzela theorem it follows that, as $\eps_k \rightarrow 0$,
 $\tilde{u}_{k}$  (up to a subsequence)  converge uniformly in $B_{1/2}$ to a H\"older
continuous function $u^*$. Also, up to a subsequence
$$G_k \to G^* \quad \text{uniformly in $C^1([0,2])$}, \quad a^-_k \to \bar a^- \in [1/2,2]$$ and hence $$a^+_k \to \bar a^+ =  G^*(\bar a^-).$$
Notice that $0 \in \Gamma(u_k)$ implies $\tilde u_k(0)=0$ hence $u^*(0)=0$.

\

\textbf{Step 2 -- Limiting Solution.} We now show that $u^*$
solves
\begin{equation}\label{Neumann}
  \begin{cases}
    {\mathcal{F}}^*(D^2u^*)=0 & \text{in $B_{1/2} \cap \{x_n \neq 0\}$}, \\
\ \\
a (u^*_n)^+ - b (u^*_n)^-=0 & \text{on $B_{1/2} \cap \{x_n =0\}$},
  \end{cases}\end{equation}
  with ${\mathcal F}^* \in \mathcal{E}(\lambda,\Lambda)$ concave, and $a=\bar a^+, b=\bar a^- {G^*}'(\bar a^-).$

Set,
\be\label{f*}
{\mathcal{F}}^*_k(N)=\frac{1}{\epsilon_k}\mathcal{F}_k(\epsilon_k N + M_k).
\ee
Then (recall that $\mathcal F_k(M_k)=0$), ${\mathcal{F}}^*_k \in \mathcal{E}(\lambda,\Lambda)$ and it is concave. Thus, up to extracting a subsequence, $${\mathcal F}^*_k \to {\mathcal F}^*, \quad \text{uniformly on compact subsets of $\mathcal S^{n \times n}.$} $$
Moreover, since $\mathcal F_k$ is homogeneous of degree 1, 
  $$
 { \mathcal{F}}^*_k(D^2\tilde{u}_k)= \mathcal F^*_k \left (\frac{1}{a_k^+ \eps_k}D^2 u_k - \frac{1}{\eps_k}M_k\right)= \frac{1}{a_k^+ \eps_k} \mathcal F_k (D^2 u_k) = 0,\quad \mbox{in} \quad B_1^+(u_k),
  $$
  and similarly ${ \mathcal{F}}^*_k(D^2\tilde{u}_k)=0$ in $B_1^-(u_k)$.
 Then, by standard arguments (see Proposition 2.9 in \cite{CC}), we conclude that 
$${\mathcal{F}}^*(D^2 u^*)=0 \quad \text{in  $B_{1/2}\cap \{x_n \neq 0\}$}.$$
Next we verify the transmission condition in the viscosity sense of Definition \ref{def_visc}. Let 
$$w(x):= A+ px_n^+-qx_n^-+ x^T N x+ \xi' \cdot x'$$
with $A \in \R$, $N$ a diagonal matrix with $\mathcal M^-(N)>0,$
and 
$$ap-bq>0.  $$
Assume that $w$ touches $u^*$ strictly by below at a point $x_0=(x_0',0) \in B_{1/2}.$  Set,
$$\bar a^+_k=a^+_k(1+\eps_k p), \quad \bar a^-_k=a^-_k(1+\eps_k q)$$
$$\bar M_k=M_k +2\eps_k N, \quad Q_k:=P_{\bar M_k, e_n} + \eps_k  \xi'\cdot x'+ A\eps_k,$$
and
$$V_k:=\bar a^+_k Q_k^+ - \bar a^-_k Q_k^-.$$

Recall that $\tilde V_k$ converges uniformly to $w$ on $B_{1/2}$ (see Remark \ref{rem_w}). Since $\tilde u_k$ converges uniformly to $u^*$ and $w$ touches $u^*$ strictly by below at $x_0$, we conclude that for a sequence of constants $c_k \to 0$ and points $x_k \to x_0$ the function
$$w_k:= V_k(x_k + \eps_k c_k e_n)$$
touches $u_k$ by below at $x_k$. Proposition \ref{sub_1} gives that $w_k$ is a strict subsolution to our free boundary problem, provided that we first choose $\delta$ small enough, so to guarantee that \eqref{sub_2} holds. We reach a contradiction as we let $k \to \infty$,
thus $u^*$ is a solution to the linearized problem \eqref{Neumann}.

\

\textbf{Step 3 -- Contradiction.} Since $\tilde u_k$ converges uniformly to $u^*$ and $u^*$ enjoys the $C^{1,\alpha}$ estimate of Theorem \ref{lineareg} we have 
\be\label{uu}|\tilde u_k - (x' \cdot \nu'+ \tilde p x_n^+ - \tilde q x_n^-)| \leq C r^{1+\alpha}, \quad x \in B_{r},\ee
with $$a\tilde p-b\tilde q=0, \quad |\nu'| \leq C.$$
Call, 
$$\bar a_k^-=a_k^-(1+\eps_k \tilde q), \quad \bar a_k^+=G_k(\bar a_k^-)=a_k^+(1+\eps_k \tilde p) + O(\eps_k^2),$$ and
$$\nu_k= \frac{e_n+\eps_k(\nu',0)}{\sqrt{1+\eps_k^2|\nu'|^2}}= e_n+\eps_k(\nu',0)+ \eps_k^2\tau, \quad |\tau|\leq C.$$
We claim that we can decompose
\be\label{barnk}M_k= \bar M_k + D_k \ee with 
$$\mathcal F_k (\bar M_k)=0, \quad \bar M_k \nu_k =0, \quad \|D_k\| \leq C \|M_k\||\nu_k-e_n|,$$
hence $\|D_k\|= O(\eps_k^{3/2}).$

Indeed, since $M_k e_n=0$ first we can decompose 
$$M_k= \tilde M_k+L_k,$$ such that $M_k \nu_k=0,$ and $$\|L_k\| \le C \|M_k\|\, |\nu_k - e_n|.$$ 
Since $\mathcal F_k(M_k)=0$ we obtain that 
$$\mathcal F_k(\tilde M_k)= O(\| L_k\|).$$
Then, using ellipticity we can decompose $\tilde M_k$ further, (here $\nu_k^\perp$ is a unit vector perpendicular to $\nu_k$)
$$\tilde M_k=\bar M_k + t_k (\nu_k^\perp \otimes \nu_k^\perp), \quad t_k =O(\mathcal F_k(\tilde M_k)), \quad \bar M_k \nu_k=0, $$
so that 
$$\mathcal F_k(\bar M_k)=0,$$
and the claim \eqref{barnk} is proved.

Let us show now that for $r = r_1$ universal, (say $\alpha_1= \alpha/2$)
$$|u_k - V_{\bar a_k^\pm, \bar M_k, \nu_k}| \leq \eps_k r^{1+\alpha_1},$$
which contradicts the fact that $u_k$ does not satisfy \eqref{flat2flat2}.
 From \eqref{uu} and the definition of $\tilde u_k$ we get that in $B_r \cap (B_1^+(u_k) \cup \Gamma(u_k))$ (we can argue similarly in the negative part)
$$|u_k -a_k^+P_{M_k,e_n} -\eps_k a^+_k(x' \cdot \nu'+ \tilde p x_n^+ - \tilde q x_n^-)| \leq 2\eps_k C r^{1+\alpha}.$$
Since in this set $x_n \geq - 3 \eps_k^{1/2}$ we conclude that for all $k$ large 
$$|u_k -a_k^+P_{M_k,e_n} -\eps_k a^+_k(x' \cdot \nu'+ \tilde p x_n)| \leq 3\eps_k C r^{1+\alpha},$$
which gives
$$|u_k - \bar a_k^+(x \cdot \nu_k + \frac 12 x^T M_k x)| \le 4\eps_k C r^{1+\alpha}.$$
Finally, from \eqref{barnk} we conclude that 
$$|u_k - V_{\bar a_k^\pm, \bar M_k, \nu_k}| \leq 5C\eps_k r^{1+\alpha},$$
from which the desired bound follows for $r_1$ small enough.
\qed

\subsection{Nonlinear Dichotomy} In this subsection, we prove the Proposition \ref{Nonlinear_intro}. 
The proof is very similar to the one of Proposition \ref{c1} above except that we use Proposition \ref{dic} for the limiting transmission problem. 

\begin{proof}
\textbf{Step 1 -- Compactness and Limiting solution.} Fix $r_0$ universal  to be specified later. Assume by contradiction that we
can find sequences $\eps_k \rightarrow 0, \delta_k \to 0$ and a sequence $u_k$ of
solutions to \eqref{fb} in $B_1$ for a sequence of operators $\mathcal F_k \in \mathcal E(\lambda,\Lambda)$ and free boundary conditions $G_k$ such that
\begin{equation*}\label{flat_kk}|u_k -P_{M_k, e_n}| \leq \eps_k \quad \text{for $x \in B_1$,  $0 \in \Gamma(u_k),$}
\end{equation*} with 
$$\|M_k\| \leq \delta_k \eps_k^{1/2}, \quad M_k e_n=0, \quad \mathcal F_k(M_k)=0,$$
but $u_k$ does not satisfy either of the alternatives \eqref{flat2} (for some small constant $c_0'$ to be specified later) or \eqref{flat3} .

Let $\tilde u_k := \tilde u_{a^\pm_k,M_k, e_n, \eps_k}$ be defined as in \eqref{utilde}.
In Step 1- Step 2 of the proof of Proposition \ref{c1} we showed that 
 as $\eps_k \rightarrow 0$, $\tilde{u}_{k}$ converge uniformly (up to a subsequence) in $B_{1/2}$ to a H\"older
continuous function $u^*$ with $u^*(0)=0$, $\|u\|_{L^\infty} \le C$ and $u^*$ solves the transmission problem
\begin{equation}\label{Neumann_2}
  \begin{cases}
    {\mathcal{F}}^*(D^2u^*)=0 & \text{in $B_{1/2} \cap \{x_n \neq 0\}$}, \\
\ \\
(u^*_n)^+ - b (u^*_n)^-=0 & \text{on $B_{1/2} \cap \{x_n =0\}$},
  \end{cases}\end{equation}
  with ${\mathcal F}^* \in \mathcal{E}(\lambda,\Lambda)$ concave, and $b>0$ bounded by universal constants.

Next we show that $u^*$ is monotone decreasing in the $e_n$ direction,
\be\label{monotone} {(u_n^*)}^\pm \leq 0.\ee
Indeed, $\tilde u_k$ satisfies $\mathcal F^*_k(D^2 \tilde u_k)=0$ in $B_1^+(u_k)$ and $\Gamma(u_k) \subset \{|x_n| \leq  \eps_k^{1/2}\}.$ In particular, if $\bar x \in \{x_n > 0\}$, for $k$ large enough
$$B:=B_{\bar x_n/2}(x) \subset B_1^+(u_k)$$
and
$$|\nabla \tilde u_k| \leq C(\bar x_n) \quad \text{in $B$}.$$  Hence, 
$$\nabla u_k(\bar x)= \nabla(P_{M,e_n} + \eps_k \tilde u_k )(\bar x) = e_n + Mx + \eps_k \nabla \tilde u_k(\bar x).$$
Using that $|\nabla u_k(\bar x)|^2 \leq 1$ and $Me_n=0$ we get 
$$1 \geq 1+ 2 \eps_k (\tilde u_k)_n + O(\eps_k^{3/2}),$$ where the constant is $O(\eps_k^{3/2})$ depends on $\bar x_n.$
By $C^{1,\alpha}$ estimates the $\tilde u_k$'s converge to $u^*$ in $C^1$ in the ball $B$.
Passing to the limit as $k \to \infty$, we get that 
$$u^*_n(\bar x) \leq 0.$$
By continuity, since $u^*$ is $C^{1,\alpha}$ up to $\{x_n=0\}$ we conclude that
$$(u^*_n)^+ \leq 0 \quad \text{in $B_1 \cap \{x_n \geq 0\}$}.$$
We argue similarly for $(u^*_n)^-.$

\medskip

{\bf Step 2 -- Contradiction.} According to Proposition \ref{dic} (since \eqref{monotone} holds and $u^*$ is bounded by a universal constant) there exist universal constants $c_0, r_0>0$ such that either of the following alternative is satisfied:
\begin{enumerate}
\item \be\label{1_1}(u^*)^+_n(0) \leq -c_0\ee 
 \item \be \label{2_2} |u^*(x)- Q(x')| \leq \frac 1 4 r_0^2, \quad \text{in $B_{r_0}$}\ee
 \end{enumerate}
 where $Q(x')$ is a quadratic polynomial in the $x'$ direction with $\|Q\| \leq C$ universal and
 \be\label{fq} \mathcal F^*(D^2Q)=0.\ee
 Since $u^*(0)=0$ we have
 $$Q(x')= \xi' \cdot x' + \frac 1 2 x^T N x, \quad Ne_n=0, \quad \mathcal F^*(N)=0.$$

If \eqref{1_1} is satisfied, by the $C^{1,\alpha}$ estimates for $u^*$ we obtain that
$$|u^* - (\xi' \cdot x' + p x_n^+- q x_n^-)| \leq C r^{1+\alpha}, \quad \text{in $B_r$}$$
with $|p|, |q|, |\xi'| \leq C$ and 
$$p - bq=0, \quad p \leq -c_0.$$
Using the convergence of the $\tilde u_k$ to $u^*$ and the definition of $\tilde u_k$ we immediately get that for $k$ large, 
$$|u_k - (P_{M_k, e_n}+ \eps_k(\xi' \cdot x'+p x_n^+- q x_n^-)|\leq C \eps_k r^{1+\alpha} \quad \text{in $B_r$}.$$
Arguing as in Step 3 of the proof of Proposition \ref{c1} (using the same notation) we conclude that 
for $r \leq r_1$, 
$$|u_k - V_{\bar a_k^\pm, \bar M_k, \nu_k}| \leq \eps_k r^{1+\alpha_1} \quad \quad \mbox{with} \quad \bar a^+=1+\eps_k p.$$
Thus, after rescaling (see the argument after the statement of Proposition \ref{c1}), Corollary \ref{maincor} gives that
$$|\nabla u_k^+ (0)| \leq 1+p\eps_k + C \eps_k r^{\alpha_1} \leq 1- \frac 12 c_0 \eps_k$$
provided that $r$ is chosen small, universal. Thus $u_k$ satisfies alternative \eqref{flat2} and we have reached a contradiction.

If \eqref{2_2} is satisfied,  using the convergence of the $\tilde u_k$ to $u^*$ and the definition of $\tilde u_k$ we get
\be\label{almost_1}|u_k - (x_n + \eps_k \xi' \cdot x' + \frac 1 2 x^T N^*_k x)| \leq \frac {1}{4} \eps_k r_0^2\ee
with
$$N^*_k=M_k+ \eps_k N, \quad N^*_k e_n=0, \quad \mathcal F_k(N^*_k)=o(\eps_k),$$
where the last equality follows from (see \eqref{f*}) 
$$\mathcal F_k(N^*_k)=\eps_k \mathcal F_k^*(N) \quad \mbox{and} \quad \mathcal F^*_k(N) \to \mathcal F^*(N)=0.$$
As before denote 
$$\nu_k= \frac{e_n+\eps_k(\xi',0)}{\sqrt{1+\eps_k^2|\xi'|^2}}= e_n+\eps_k(\xi',0)+ \eps_k^2\tau, \quad |\tau|\leq C.$$

We argue as in Step 3 of Proposition \ref{c1} and decompose
$$N^*_k= \bar M_k+D_k, \quad \bar M_k  \nu_k=0, \quad \|D_k\|=o(\eps_k)$$
with
$$\mathcal F_k(\bar M_k)=0.$$
Thus, \eqref{almost_1} yields for $k$ large,
$$|u_k - P_{\bar M_k, \nu_k}| \leq \frac 1 2 \eps_k r_0^2 \leq \eps_k r_0 ^{2+\alpha_0}$$
for $\alpha_0$ small enough, and again we reached a contradiction.

\end{proof}

\subsection{Proof of Theorem \ref{T2.1}} Here we only sketch the proof and we outline the differences between the case of quasilinear equations 
\be\label{aij}
Q(u):=a_{ij}\left(\frac{\nabla u}{|\nabla u|} \right) \, \, u_{ij}=0,\ee
and the fully nonlinear concave equations case treated in Theorem \ref{T2}. 

We need to check that Lemma \ref{uxn}, Corollary \ref{eeta}, and Propositions \ref{Nonlinear_intro} and \ref{c1} apply to the quasilinear setting. 

The proofs of Lemma \ref{uxn} and Corollary \ref{eeta} are identical since Lipschitz solutions to \eqref{aij} in $B_1$ satisfy interior $C^2$ estimates. This means that in any compact set of $B_1$, directional derivatives $u_e$ solve linear elliptic equations (with first order terms) and bounded coefficients. 

For the proofs of Propositions \ref{Nonlinear_intro} and \ref{c1} we introduce the approximate quadratic polynomials in this setting. Let 
$$P_{M,\nu}:= x \cdot \nu + \frac 1 2 x^T M x,$$
such that $M$ and $\nu$ satisfy the compatibility conditions 
$$a_{ij}(\nu)m_{ij}=0, \quad \quad M \nu=0.$$
Also let $V_{a^\pm,M,\nu}$ and $\tilde u_{a^\pm,M,\nu}$ be defined as in the beginning of Section 5. One can argue as in the proofs above and check that as $\delta,\eps \to 0$ a sequence of $\tilde u$'s converges uniformly to a solution $u^*$ of the transmission problem \eqref{Neumann} with $$\mathcal F^*(D^2 u^*)=a_{ij}(\nu) u^*_{ij}=0.$$
Since $\mathcal F^*$ is linear, the results Theorem \ref{lineareg} and Proposition \ref{dic} from Section 4 apply, and the rest of the arguments follow as in the proofs above.

\qed
 
We remark that the proof of Theorem \ref{T2.1} above is easier than the one of Theorem \ref{T2}. This is because the key results for the transmission problem in Section 4 are straightforward in the case of linear operators $\mathcal F$. 
In fact it can be shown that solutions to a linear transmission problem are of class $C^{\infty}(B_1^\pm)$. It follows that the two-phase problem governed by \eqref{aij} enjoys a quadratic 
improvement of flatness property and then Propositions \ref{Nonlinear_intro} and \ref{c1} can be easily deduced.

\end{document}